\documentclass [smallextended]{svjour3}

\usepackage{natbib}
\usepackage{amsmath}
\usepackage{amssymb}
\usepackage{url}

\usepackage{subfig}
\usepackage{siunitx}
\usepackage[ruled,vlined]{algorithm2e}

\newtheorem{thm}{Theorem}[section]

\newtheorem{lem}[thm]{Lemma}

\newtheorem{defn}[thm]{Definition}

\newtheorem{cor}[thm]{Corollary}

\newcommand{\ud}{\,\mathrm{d}}

\begin{document}

\title{Multiscale methods with compactly supported radial basis functions for Galerkin approximation of elliptic PDEs}

\titlerunning{Multiscale methods for Galerkin approximation of elliptic PDEs}

\author{A. Chernih \and Q.T. Le Gia }

\institute{A. Chernih
\at School of Mathematics and Statistics, University of New South Wales, Sydney NSW 2052,
Australia\\Tel.: +61-410-697411, Fax: +612 93857123 \\\email{andrew@andrewch.com}
\and
Q.T. Le Gia
\at School of Mathematics and Statistics, University of New South Wales, Sydney NSW 2052,
Australia\\\email{qlegia@unsw.edu.au}
}

\maketitle
\begin{center}
\today
\end{center}

\begin{abstract}
The aim of this work is to consider multiscale algorithms for solving PDEs with Galerkin methods on bounded domains. We provide results on convergence and condition numbers. We show how to handle PDEs with Dirichlet boundary conditions. We also investigate convergence in terms of the mesh norms and the angles between subspaces to better understand the differences between the algorithms and the observed results. We also consider the issue of the supports of the RBFs overlapping the boundary in our stability analysis, which has not been considered in the literature, to the best of our knowledge.
\end{abstract}




\section{Introduction}

Radial basis functions (RBFs) have been increasingly important in the area of approximation theory. For solving partial differential equations (PDEs), RBFs have been studied more for meshless collocation \citep{GieW06,Fas07}, however their use with Galerkin methods have also been considered \citep{Wen99,Wen98b}. Two excellent recent books covering practical and theoretical issues are \cite{Fas07} and \cite{Wen05}. A function $\Phi : \mathbb{R}^d \rightarrow \mathbb{R}$ is said to be \textit{radial} if there exists a function $\phi: [0,\infty) \rightarrow \mathbb{R}$ such that $\Phi(\mathbf{x}) = \phi(\|\mathbf{x}\|_2)$ for all $\mathbf{x} \in \mathbb{R}^d$, where $\|\cdot\|_2$ denotes the usual Euclidean norm in $\mathbb{R}^d$.  Then we can define an RBF for a given centre $\mathbf{x}_i \in \mathbb{R}^d$ as
\begin{equation*}
\Phi(\mathbf{x}_i) := \phi(\|\mathbf{x}-\mathbf{x}_i\|_2).
\end{equation*}

A practical issue that arises is that of which scale to use for the radial basis functions. A small scale will lead to a sparse and consequently well-conditioned linear system, but at the price of poor approximation power. Conversely, a large scale will have better approximation power but at the price of an ill-conditioned linear system.

Many examples may naturally exhibit multiple scales, for example, constructing an approximation for the height of the earth's surface may suggest a ``large scale" to be used over desert regions and a ``fine scale" over areas of high variability, such as the Himalayas. It appears much more appropriate to allow different scales in different regions. Of course, this comes at the price of having to select which scales to use in which regions but this is not the topic of this paper.

The multiscale algorithms proposed in this paper are constructed over multiple levels, in which the residual of the current stage is the target function for the next stage, and at each stage, RBFs with smaller support and with more closely spaced centres will be used as basis functions.

Such a multiscale algorithm for interpolation was first proposed in \cite{FloI96} and \cite{Sch96} but without any theoretical grounding. Theoretical convergence was proven in the case of the data points being located on a sphere \citep{LeSW10} and then extended to interpolation and approximation on bounded domains \citep{Wen10}.

In \cite{Wen98b} we can find various experiments with two multiscale algorithms for constructing Galerkin approximations to PDEs on bounded domains. These two algorithms are studied in this paper. There was no proof of convergence for the first algorithm which we call the \textit{multiscale algorithm}. For the second multiscale algorithm, which we call the \textit{nested multiscale algorithm}, convergence was proven, making use of the fact that the weak formulation of a PDE can be interpreted as a Hilbert space projection method. Numerical experiments with both multiscale algorithms were also given.

The aim of this work is to consider multiscale algorithms for solving PDEs with Galerkin methods on bounded domains. We provide results on convergence and condition numbers. We show how to handle PDEs with Dirichlet boundary conditions which was not covered in the earlier papers. We also investigate convergence in terms of the mesh norms and the angles between subspaces to better understand the differences between the two algorithms and the observed results. We also consider the issue of the supports of the RBFs overlapping the boundary in our stability analysis, which has not been considered in the literature, to the best of our knowledge.

Related multilevel methods for the numerical solution of PDEs include domain decomposition \citep{SmiBG96} and the use of smoothing as a postconditioner with an iterative (multilevel) method \citep{Fas99}. The heterogeneous multiscale method \citep{WeiELRV07} differs in that it is multiscale in the time domain, whilst we work with multiple scales in the spatial domain. Multigrid methods \citep{BreS08} require the use of a computational mesh, whilst our approach is meshfree.

In the next section we review the background material required for the multiscale algorithms. Sections \ref{SectionPDEsNoDirchlt} and \ref{SectionPDEsDirchlt} describe how to construct Galerkin solutions to PDEs using radial basis functions for PDEs without and with Dirichlet boundary conditions respectively. Sections \ref{sectionMLgalerkin} and \ref{SectionNestedML} consider the multiscale and nested multiscale algorithms. Section \ref{SectionNumericalExperiments} provides the results of several examples using the two multiscale algorithms. Finally Section \ref{SectionAlphaAnalysis} provides an analysis of the convergence of the two algorithms.

\section{Preliminaries} \label{SectionPreliminaries}

In this paper, we will use (scaled) compactly supported radial basis functions to construct multiscale approximate solutions to PDEs, that is, we form the solution over multiple levels. We will work with a given domain $\Omega \subseteq \mathbb{R}^d$. A kernel $\Phi: \Omega \times \Omega \rightarrow \mathbb{R}$ is also given.

At each level, we will have a finite point set $X \subseteq \Omega$. We will define the \textit{fill distance} as
$$h_{X, \Omega} := \sup_{\mathbf{x} \in \Omega} \min_{\mathbf{x}_j \in X} \|\mathbf{x}-\mathbf{x}_j\|_2,$$
which is a measure of the uniformity of the points in $X$ with respect to $\Omega$. Then for example, at each level $i$, we denote the fill distance by $h_i$. The selection of point sets with fill distances decreasing in a specific way will form one of the requirements for convergence of our algorithms.

The functions that we will be concerned with are defined on a bounded domain $\Omega$ with a Lipschitz boundary. As a result, there is an extension operator for functions defined in Sobolev spaces which is presented in the following lemma \cite[Theorem 1.4.5]{BreS08}. For further details, we refer the reader to \cite{Ste70}.
\begin{lem}
\label{lemExtensionOperator}
Suppose $\Omega \subseteq \mathbb{R}^d$ has a Lipschitz boundary. Then there is an extension mapping $E: H^{\tau}(\Omega) \rightarrow H^{\tau}(\mathbb{R}^d)$, defined for all non-negative integers $\tau$, satisfying $Ev\vert_{\Omega} = v \mbox{ for all } v \in H^{\tau}(\Omega)$ and
\begin{equation*}
\|Ev\|_{H^{\tau}(\mathbb{R}^d)} \leq C \|v\|_{H^{\tau}(\Omega)}.
\end{equation*}
\end{lem}
In this paper, $C$ will denote a generic constant.

Since we also have $\|v\|_{H^{\tau}(\Omega)} \leq \|Ev\|_{H^{\tau}(\mathbb{R}^d)}$, this means that when we need to consider the $H^{\tau}(\Omega)$ norms of the errors at each level, we can carry out our error analysis in the $H^{\tau}(\mathbb{R}^d)$-norm. This is advantageous, since we then have for $g \in H^{\tau}(\mathbb{R}^d)$
\begin{equation}
\label{eqnHtauNorm}
\|g\|_{H^{\tau}(\mathbb{R}^d)}^2 = \int\limits_{\mathbb{R}^d} |\widehat{g}(\boldsymbol{\omega})|^2 \left(1 + \|\boldsymbol{\omega}\|_2^2 \right)^{\tau} \mathrm{d}\boldsymbol{\omega},
\end{equation}
upon defining the Fourier transform as
$$\widehat{g}(\boldsymbol{\omega}) = \left(2\pi\right)^{-d/2} \int\limits_{\mathbb{R}^d} g(\boldsymbol{x}) e^{-i\mathbf{x}^T\boldsymbol{\omega}} \mathrm{d}\mathbf{x}.$$
At each level, we will also require a scaled version of the kernel $\Phi: \Omega \times \Omega \rightarrow \mathbb{R}$. For our unscaled kernel we will use a Wendland compactly supported radial basis function \citep{Wen05}. With a (level-specific) scaling parameter $\delta>0$, we can define the scaled kernels as
\begin{equation}
\Phi_{\delta}(\mathbf{x},\mathbf{y}) := \delta^{-d} \phi\left( \frac{\|\mathbf{x}-\mathbf{y}\|_2}{\delta}\right). \label{eqnDefnScaledKernel}
\end{equation}
Appropriate selection of the scaling parameters will also prove to be one of the important ingredients for convergence of our multiscale algorithms.

For the Wendland basis functions, there exist two constants $0 < c_1 \leq c_2$ such that their Fourier transforms satisfy \citep{Wen05}
\begin{equation}
\label{eqnWendlandFourierTequiv}
c_1\left(1 + \|\boldsymbol{\omega}\|_2^2 \right)^{-\tau} \leq \widehat{\Phi}(\boldsymbol{\omega}) \leq c_2 \left(1 + \|\boldsymbol{\omega}\|_2^2 \right)^{-\tau}, \quad \boldsymbol{\omega} \in \mathbb{R}^d,
\end{equation}
where $\tau = (d+2k+1)/2$ where $d$ is the spatial dimension and $k$ is the smoothness parameter. For further details, we refer the reader to \cite{Wen05}. The native space $\mathcal{N}_{\Phi}(\mathbb{R}^d)$ of $\Phi$ consists of all functions $g \in L_2(\mathbb{R}^d)$ such that
\begin{equation}
\label{eqnPhiNorm}
\|g\|_{\Phi}^2 = \int\limits_{\mathbb{R}^d} \frac{|\widehat{g}(\boldsymbol{\omega})|^2}{\widehat{\Phi}(\boldsymbol{\omega})} \mathrm{d}\boldsymbol{\omega} < \infty.
\end{equation}
Taking \eqref{eqnHtauNorm} and \eqref{eqnWendlandFourierTequiv} together shows that $\mathcal{N}_{\Phi}(\mathbb{R}^d)$ is norm-equivalent to the Sobolev space $H^{\tau}(\mathbb{R}^d)$.

Consequently the Fourier transform of $\Phi_{\delta}$, $\widehat{\Phi_{\delta}}(\boldsymbol{\omega}) = \widehat{\Phi}(\delta \boldsymbol{\omega})$, satisfies
\begin{equation}
\label{eqnScaledWendlandFourierTequiv}
c_1\left(1 + \delta^2\|\boldsymbol{\omega}\|_2^2 \right)^{-\tau} \leq \widehat{\Phi_{\delta}}(\boldsymbol{\omega}) \leq c_2 \left(1 + \delta^2\|\boldsymbol{\omega}\|_2^2 \right)^{-\tau}, \quad \boldsymbol{\omega} \in \mathbb{R}^d.
\end{equation}
We will need norm equivalence as stated in the following lemma.
\begin{lem}
\label{lemPhiDeltaHtauNormEquiv}
For every $\delta \in (0,\delta_a]$ and for all $g \in H^{\tau}(\mathbb{R}^d)$, there exist constants $0 < c_3 \leq c_4$ such that
\begin{equation*}
c_3 \|g\|_{\Phi_{\delta}} \leq \|g\|_{H^{\tau}(\mathbb{R}^d)} \leq c_4 \delta^{-\tau} \|g\|_{\Phi_{\delta}}.
\end{equation*}
\end{lem}
\begin{proof}
The case $\delta_a \leq 1$ was proven in \cite{Wen10} with $c_3 = c_1^{1/2}$ and $c_4 = c_2^{1/2}$. To extend this to the case where $\delta_a > 1$, note that for $\delta > 1$ we have
$$ \left(1 + \|\boldsymbol{\omega}\|_2^2 \right)^{\tau} = \delta^{-2\tau} \left(\delta^2 + \delta^2\|\boldsymbol{\omega}\|_2^2 \right)^{\tau} \geq \delta_a^{-2\tau} \left(1 + \delta^2\|\boldsymbol{\omega}\|_2^2 \right)^{\tau}.$$
Together with \eqref{eqnHtauNorm}, \eqref{eqnScaledWendlandFourierTequiv} and \eqref{eqnPhiNorm} we can see that
\begin{eqnarray*}
\|g\|_{H^{\tau}(\mathbb{R}^d)}^2 &=& \int\limits_{\mathbb{R}^d} |\widehat{g}(\boldsymbol{\omega})|^2\left(1 + \|\boldsymbol{\omega}\|_2^2 \right)^{\tau} \mathrm{d}\boldsymbol{\omega} \\
&\geq& \delta_a^{-2 \tau} \int\limits_{\mathbb{R}^d} |\widehat{g}(\boldsymbol{\omega})|^2\left(1 + \delta^2\|\boldsymbol{\omega}\|_2^2 \right)^{\tau} \mathrm{d}\boldsymbol{\omega} \\
&\geq& c_1 \delta_a^{-2\tau} \int\limits_{\mathbb{R}^d} \frac{|\widehat{g}(\boldsymbol{\omega})|^2}{\widehat{\Phi_{\delta}}(\boldsymbol{\omega})} \mathrm{d}\boldsymbol{\omega} \\
&\geq& c_1 \delta_a^{-2\tau} \|g\|_{\Phi_{\delta}}^2.
\end{eqnarray*}
For the lower bound, we can just use $\delta > 1$ directly to derive
\begin{eqnarray*}
\|g\|_{H^{\tau}(\mathbb{R}^d)}^2 &=& \int\limits_{\mathbb{R}^d} |\widehat{g}(\boldsymbol{\omega})|^2\left(1 +  \|\boldsymbol{\omega}\|_2^2 \right)^{\tau} \mathrm{d}\boldsymbol{\omega} \\
&\leq& \int\limits_{\mathbb{R}^d} |\widehat{g}(\boldsymbol{\omega})|^2\left(1 + \delta^2\|\boldsymbol{\omega}\|_2^2 \right)^{\tau} \mathrm{d}\boldsymbol{\omega} \\
&\leq& c_2\|g\|_{\Phi_{\delta}}^2,
\end{eqnarray*}
using \eqref{eqnScaledWendlandFourierTequiv} and \eqref{eqnPhiNorm}. Then setting $c_3 := c_1^{1/2} \min(1,\delta_a^{-\tau})$ and $c_4 := c_2^{1/2} \max(1,\delta_a^{\tau})$ completes the proof. \hspace*{\fill} \qed
\end{proof}

\section{PDEs with Neumann and/or Robin boundary conditions} \label{SectionPDEsNoDirchlt}

In this section, we consider a second order PDE which has homogeneous Neumann and/or Robin boundary conditions. For example, such a PDE with Neumann boundary conditions is given by
\begin{subequations}
\label{eqnPDEnonDirchlt}
\begin{align}
\mathcal{L} u &= f \quad \mbox{ in } \Omega,  \\
\frac{\partial u}{\partial \mathbf{n}} &= 0 \quad \mbox{ on } \partial \Omega,
\end{align}
\end{subequations}
where $\mathcal{L}$ is a second order elliptic differential operator, $\mathbf{n}$ denotes the outward unit normal vector and $\partial \Omega$ denotes the boundary.
The weak formulation is given by
\begin{equation}
a(u,v) = \left\langle f,v \right\rangle_{L_2(\Omega)} \quad \forall \, v \in V, \label{eqnWeakFormPDE}
\end{equation}
where $V = H^1(\Omega)$. We assume that $\mathcal{L}$ and $f$ are such that $a(u,v)$ is a strictly coercive and continuous bilinear form defined on $V \times V$ and $\left\langle f,v \right\rangle_{L_2(\Omega)}$ is a continuous linear form defined on $V$. By the Lax-Milgram theorem, \eqref{eqnWeakFormPDE} has a unique solution $u \in V$. We will also require $u \in H^{2}(\Omega)$ with spatial dimension $d \leq 3$.

Galerkin approximation seeks to solve \eqref{eqnWeakFormPDE} with a finite dimensional subspace $V_N\subseteq V$. In other words, the Galerkin approximation $\widetilde{u}_N$ is the solution of
\begin{equation}
\widetilde{u}_N \in V_N: a(\widetilde{u}_N,v) = \left\langle f,v \right\rangle_{L_2(\Omega)} \quad \forall \, v \in V_N. \label{eqnGalerkin}
\end{equation}
We will consider $\Omega$ to be a bounded domain with a Lipschitz boundary, which means that we can apply the extension operator to use norms in $\mathbb{R}^d$ as explained in the introduction. For further information on weak formulation of PDEs and Galerkin approximation, we refer the reader to \cite{BreS08}.

Since the PDE does not have Dirichlet boundary conditions, we can use the entire Sobolev space $H^{1}(\Omega)$ rather than the subspace $\mathring{H}^{1}(\Omega)$ consisting of functions with zero boundary values which can occur with pure Dirichlet boundary conditions.

We will consider finite dimensional subspaces $V_N \subseteq V$ of the form
\begin{equation*}
V_N := \mbox{span}\left\{\Phi\left(\cdot,\mathbf{x}_j\right): 1\leq j \leq N \right\}
\end{equation*}
where $\Phi: \mathbb{R}^d \rightarrow \mathbb{R}$ is at least a $C^1$-function and there are $N$ centres $\{\mathbf{x}_j: 1\leq j \leq N\}$. In this case our approximation takes the form
\begin{equation*}
\widetilde{u}_N = \sum_{j=1}^N c_j \Phi(\cdot, \mathbf{x}_j),
\end{equation*}
and the weak formulation with this approximation given by
$$a(\widetilde{u}_N,v) = \left\langle f,v \right\rangle_{L_2(\Omega)} \quad \forall \, v \in V_N,$$ results in a linear system $$\mathbf{A}\mathbf{c} = \mathbf{f}$$ where the entries of the \textit{stiffness matrix} are given by
\begin{equation}
A_{ij} = a(\Phi(\cdot,\mathbf{x}_i),\Phi(\cdot,\mathbf{x}_j)) \label{eqnAgalerkin}
\end{equation}
 and $$\mathbf{f}_i = \int\limits_{\Omega} f(\mathbf{x})   \Phi(\mathbf{x},\mathbf{x}_i) \, \mathrm{d}\mathbf{x}.$$

We have the following result from \cite{Wen99}.
\begin{thm}
\label{thmGalerkinErrorBoundH1_H1}
If $u \in H^{2}(\Omega), d \leq 3,$ is the solution to the variational problem \eqref{eqnWeakFormPDE} and $\widetilde{u}_N \in V_N$ is the solution of \eqref{eqnGalerkin}, where $V_N$ is generated with $X$ satisfying $h\leq h_0$ for $h_0$ small enough and fixed, then the error can be bounded by
$$ \| u-\widetilde{u}_N\|_{H^1(\Omega)} \leq C h\|u\|_{H^{2}(\Omega)}.$$
\end{thm}

\begin{lem}
\label{lemGalerkinL2errorBound}
Consider the adjoint variational problem
\begin{equation}
a(v,u) = \left\langle f,v \right\rangle_{L_2(\Omega)} \quad \forall \, v \in V. \label{eqnGalerkinAdjoint}
\end{equation}
If the solution $u$ satisfies the regularity estimate
\begin{equation}
\|u\|_{H^2(\Omega)} \leq C\|f\|_{L_2(\Omega)}
\label{eqnRegEstimateL2}
\end{equation}
for both the variational problem \eqref{eqnWeakFormPDE} and the adjoint variational problem \eqref{eqnGalerkinAdjoint}, then we have the following error bound
$$ \|u-\widetilde{u}_N\|_{L_2(\Omega)} \leq C h \|u-\widetilde{u}_N\|_{H^1(\Omega)}\leq C h^{2} \|u\|_{H^{2}(\Omega)}.$$
\end{lem}
\begin{proof}
We follow a duality argument as in \cite[Theorem 5.7.6]{BreS08}. With $\widetilde{u}_N$ as defined in \eqref{eqnGalerkin}, let $w$ be the solution to the adjoint problem
\begin{equation*}
a(v,w) = \left\langle u - \widetilde{u}_N,v \right\rangle_{L_2(\Omega)} \quad \forall \, v \in V,
\end{equation*}
and let the Galerkin approximation be given by $\widetilde{w}$. Then since the bilinear form $a(\cdot,\cdot)$ is bounded, $u \in H^2(\Omega)$, and with Theorem \ref{thmGalerkinErrorBoundH1_H1} we have
\begin{eqnarray*}
\|u-\widetilde{u}_N\|_{L_2(\Omega)}^2 &=& \left\langle u-\widetilde{u}_N, u-\widetilde{u}_N \right\rangle_{L_2(\Omega)} \\
&=& a(u-\widetilde{u}_N,w) = a(u-\widetilde{u}_N,w-\widetilde{w}) \\
&\leq& C\|u-\widetilde{u}_N\|_{H^1(\Omega)} \, \|w-\widetilde{w}\|_{H^1(\Omega)} \\
&\leq& Ch\|u-\widetilde{u}_N\|_{H^1(\Omega)} \, \|w\|_{H^2(\Omega)} \\
&\leq& Ch\|u-\widetilde{u}_N\|_{H^1(\Omega)} \, \|u-\widetilde{u}_N\|_{L_2(\Omega)},
\end{eqnarray*}
where in the second last line, we use the regularity estimate \eqref{eqnRegEstimateL2} and the result follows with another application of Theorem \ref{thmGalerkinErrorBoundH1_H1}. \hspace*{\fill} \qed
\end{proof}
We note that \eqref{eqnRegEstimateL2} is known to hold \cite[p.139]{BreS08}
 \begin{itemize}
 \item if $\Omega$ has a smooth boundary and the problem has pure Dirichlet or pure Neumann boundary conditions;
 \item if $d=2$ and $\Omega$ is convex and the problem has pure Dirichlet or pure Neumann boundary conditions.
 \end{itemize}

\section{PDEs with Dirichlet boundary conditions} \label{SectionPDEsDirchlt}

In this section we will consider a PDE with Dirichlet boundary conditions and we seek error estimates similar to those in the previous section. For ease of description, we will consider the following boundary value problem
\begin{subequations}
\label{eqnPDEdirchlt}
\begin{align}
-\Delta u &= f \quad \mbox{ in } \Omega,  \\
u &= g \quad \mbox{ on } \partial \Omega,
\end{align}
\end{subequations}
where $\Delta$ denotes the Laplacian in $\mathbb{R}^d$.

\cite{Nit70} proposed minimising the functional $J(v-u)$ where

\begin{equation*}
J(w) := \int\limits_{\Omega} \vert \nabla w \vert^2 - 2\int\limits_{\partial \Omega}w \left(\nabla w \cdot \mathbf{n} \right) + \beta_N \int\limits_{\partial \Omega} w^2,
\end{equation*}
for all $v \in V_N$ with $u$ being the solution of \eqref{eqnPDEdirchlt} and where $\mathbf{n}$ denotes the outward unit normal vector and $\nabla$ the gradient operator. The parameter $\beta_N >0$ depends only on the subspace $V_N$. The approximation $\widetilde{u}_N$ is given by $J(\widetilde{u}_N - u) := \inf_{v \in V_N} J(v-u).$ With $f$ and $g$ from \eqref{eqnPDEdirchlt}, we can compute $\widetilde{u}_N$ since
 \begin{equation*}
 J(v-u) = J(v) + J(u) - 2 \left( \int\limits_{\Omega} fv + \int\limits_{\partial \Omega} g \left( \beta_N v - \nabla v \cdot \mathbf{n} \right) \right).
 \end{equation*}
Then the variational form to approximate \eqref{eqnPDEdirchlt} becomes: find $\widetilde{u}_N \in V_N$ such that for all $v \in V_N$
\begin{equation}
a_D(\widetilde{u}_N,v) = \ell_D(v),
\end{equation}
where we use the subscript $D$ to denote the Dirichlet boundary conditions and
\begin{subequations}
\label{eqnNitsche}
\begin{align}
a_D(u,v) &:= \int\limits_{\Omega} \nabla u \cdot \nabla v - \int\limits_{\partial \Omega} v \left(\nabla u \cdot \mathbf{n} \right) - \int\limits_{\partial \Omega} u \left(\nabla v \cdot \mathbf{n} \right)
+ \beta_N \int\limits_{\partial \Omega} uv \\
\ell_D(v) &:= \int\limits_{\Omega} fv - \int\limits_{\partial \Omega} g \left(\nabla v \cdot \mathbf{n} \right) + \beta_N \int\limits_{\partial \Omega} vg.
\end{align}
\end{subequations} It can be shown that the variational form using Nitsche's method leads to variational consistency, in the sense that if $u$ is sufficiently regular, then \cite[p. 119]{BloCS03}
\begin{equation*}
a_D(u,v)= \ell_D(v), \quad \forall \,\,v \in V_N.
\end{equation*}
If there exists a positive constant $C_N$ such that
\begin{equation}
\label{eqnNitscheBound}
\| \nabla v \cdot \mathbf{n} \|_{L_2(\partial \Omega)} \leq \frac{C_N}{\sqrt{\delta}} \| \nabla v \|_{L_2(\Omega)} \quad \forall \, v \, \in V_N ,
\end{equation}
with $\delta$ being the support of the radial basis functions, then selecting
\begin{equation}
\label{eqnNitscheBeta}
\beta_N = \frac{c_5}{\delta}
\end{equation}
with $c_5 > 2C_N^2$ will ensure that the bilinear form $a_D(\cdot,\cdot)$ is symmetric positive definite. This choice of $c_5$ will also ensure that $a_D$ is coercive since
\begin{eqnarray*}
a_D(v,v) &=& \| \nabla v \|^2_{L_2(\Omega)} - 2\int\limits_{\partial \Omega}v \left(\nabla v \cdot \mathbf{n} \right) + \beta_N \|v\|^2_{L_2(\partial \Omega)} \\
&\geq& \| \nabla v \|^2_{L_2(\Omega)} - \frac{2C_N}{\delta^{1/2}} \|v\|_{L_2(\partial \Omega)} \|\nabla v\|_{L_2(\Omega)} + \beta_N \|v\|^2_{L_2(\partial \Omega)} \\
&\geq& \frac{1}{2} \| \nabla v \|^2_{L_2(\Omega)} + \left(\beta_N - \frac{2C_N^2}{\delta} \right)\|v\|^2_{L_2(\partial \Omega)} \\
&\geq& C \|v\|^2_{H^1(\Omega)}, \forall \, v \, \in V_N
\end{eqnarray*}
where we have used the Cauchy-Schwarz and Friedrichs inequalities, \eqref{eqnNitscheBound} and that $2xy \leq x^2 + y^2$. We recall that the Friedrichs inequality \citep{Maz11} states that
 \begin{equation*}
 \int\limits_{\Omega} \vert u \vert^2 \leq C \left(\int\limits_{\Omega} \vert \nabla u \vert^2 + \int\limits_{\partial \Omega} \vert u \vert^2 \right),
 \end{equation*}
 for $\Omega$ being a bounded domain for which the Gauss-Green formula holds.

Continuity follows since
\begin{eqnarray*}
\vert a_D(u,v) \vert &\leq& \vert u \vert_{H^1(\Omega)} \vert v \vert_{H^1(\Omega)} + C_N/\delta \left( \|v\|_{L_2(\partial \Omega)}\|\nabla u\|_{L_2(\Omega)} + \right. \\
&& \left. \|u\|_{L_2(\partial \Omega)}\|\nabla v\|_{L_2(\Omega)} \right) + \beta_N \|u\|_{L_2(\partial \Omega)} \|v\|_{L_2(\partial \Omega)} \\
&\leq& \vert u \vert_{H^1(\Omega)} \vert v \vert_{H^1(\Omega)} + C \left( \|v\|_{L_2(\Omega)}\|\nabla u\|_{L_2(\Omega)} + \right. \\
&& \left. \|u\|_{L_2(\Omega)}\|\nabla v\|_{L_2(\Omega)} \right) + \beta_N \|u\|_{L_2(\Omega)} \|v\|_{L_2(\Omega)} \\
&\leq& C \| u \|_{H^1(\Omega)} \| v \|_{H^1(\Omega)}, \quad \forall \, u,v \, \in V_N
\end{eqnarray*}
where we have used the Cauchy-Schwarz inequality, \eqref{eqnNitscheBound} and the Sobolev trace embedding theorem.

Nitsche also proved that the optimal error estimates of Theorem \ref{thmGalerkinErrorBoundH1_H1} and Lemma \ref{lemGalerkinL2errorBound} hold in this setting if, in addition to the requirement of selecting $\beta_N$ satisfying \eqref{eqnNitscheBeta}, there exists a $s_u \in V_N$ such that for $u \in H^2(\Omega)$, the following error bounds hold for $k \in \{0,1\}$
\begin{subequations}
\label{eqnNitscheReqErrorBounds}
\begin{align}
\|u-s_u\|_{H^k(\Omega)} \leq C h^{2-k} \|u\|_{H^2(\Omega)}, \\
\|u-s_u\|_{H^k(\partial \Omega)} \leq C h^{3/2-k} \|u\|_{H^2(\Omega)}.
\end{align}
\end{subequations}
With our choice of compactly supported radial basis functions, this requirement is known to hold \citep{Wen99}.

Note that the most challenging aspect of Nitsche's method is the derivation of the weak form and the selection of the stabilisation parameter $\beta_N$. Both the weak form and the choice of the parameter $\beta_N$ depend on the PDE as well as the Dirichlet boundary conditions.

\section{Multiscale Galerkin approximation} \label{sectionMLgalerkin}

In this section, we consider a multiscale algorithm for constructing a Galerkin approximation where we use the residual from the previous level as the target for each subsequent level. We define the approximation at level $i$ as $\widetilde{u}_i := \widetilde{u}_{N_i}$ with centres $N_i$ and the approximation space at level $i$ as $V_i := V_{N_i}$. The algorithm is given in Algorithm \ref{AlgGalerkin}. The bilinear form $a(\cdot,\cdot)$ used in this algorithm is the unmodified bilinear form in the case of a PDE with Neumann or Robin boundary conditions and the Nitsche's method bilinear form $a_D(\cdot,\cdot)$ in the case of a PDE with Dirichlet boundary conditions.

\begin{algorithm}[!htbp]
\DontPrintSemicolon
{\normalsize
\KwData{\hspace{0.45in} $n$: number of levels \\ \hspace{0.45in} $\{X_i\}_{i=1}^{n}$: the set of nested centres for each level $i$, with mesh \\ \hspace{0.45in} norms at each level given by $h_i$ satisfying $c \mu h_i \leq h_{i+1} \leq \mu h_i$ \\ \hspace{0.45in} with fixed $\mu \in (0,1), c \in (0,1]$ and $h_1$ sufficiently small \\ \hspace{0.45in} $\{\delta_i\}_{i=1}^n:$ the scale parameters to use at each level, \hspace{0.45in} \\ \hspace{0.45in} satisfying $\delta_i = \nu h_i$, $\nu$ a fixed constant.}

\Begin{
Set $\widetilde{u}_0=0$ \\
\For{$i =1,2,\ldots,n$}{
With the level-specific approximation subspace $V_i := \mbox{span}\left\{\Phi_{\delta_i}(\cdot,\mathbf{x}), \mathbf{x} \in X_i\right\}$ solve the Galerkin approximation given by
\begin{equation*}
\mbox{Find } s_i \in V_i \, : \, a(s_i,v) = \left\langle f,v \right\rangle_{L_2(\Omega)} - a(\widetilde{u}_{i-1},v) \, \, \forall \, v \in V_i
\end{equation*}
Update the solution according to
\begin{eqnarray*}
\widetilde{u}_i &=& \widetilde{u}_{i-1} + s_i \\
\end{eqnarray*}
}
}
\KwResult{Approximate solution at level $n$, $\widetilde{u}_n$ \\ \hspace{0.45in} The error at level $n$, $e_n := u - \widetilde{u}_n$.}
}
\caption{Multiscale Galerkin approximation \label{AlgGalerkin}}
\end{algorithm}

The algorithm as stated uses the same bilinear form at each level and it is the approximation space $V_i$ which changes. However the Nitsche's method bilinear form $a_D(\cdot,\cdot)$ will vary at each level since the value of $C_N^2$ is proportional to $\delta^{-1}$ and hence so is $\beta_N$. This means that we will need to select the value of $\beta_N$ corresponding to the last level and to use this for all previous levels. This will also mean that we will need to know the number of levels in advance.

Henceforth we will simply refer to the bilinear form as $a(\cdot,\cdot)$. This should cause no confusion as we have the same error bounds in both cases, as well as coercivity and continuity, and the multiscale algorithm follows the same steps in both cases. We require one more lemma before we can analyse the convergence of the multiscale algorithm.

\begin{lem}
\label{lemGalerkinInterStageErrorBound}
Let $\Omega \subseteq \mathbb{R}^d$ be a bounded domain with a Lipschitz boundary. Then for Algorithm \ref{AlgGalerkin} and for a given level $i>1$, we have the following bound on the $H^1(\Omega)$ error between subsequent levels.
$$\|e_i\|_{H^{1}(\Omega)} \leq C \|e_{i-1}\|_{H^1(\Omega)},$$
where $e_i$ is defined in Algorithm \ref{AlgGalerkin}.
\end{lem}
\begin{proof}
We will firstly show that $s_i$ is the Galerkin approximation of $e_{i-1}$. We have
\begin{eqnarray*}
a(s_i,w) &=& \left\langle f,w \right\rangle_{L_2(\Omega)} - a(\widetilde{u}_{i-1},w), \,\, w \in V_i \\
&=& a(u,w) - a(\widetilde{u}_{i-1},w) \\
&=& a(u-\widetilde{u}_{i-1},w) \\
&=& a(e_{i-1},w),
\end{eqnarray*}
where we have used the variational form of the PDE and the linearity in the first argument of the bilinear form $a(\cdot,\cdot)$. Hence on setting $w = s_i$ we obtain
\[
a(e_{i-1}-s_i,s_i)=0.
\]
Upon noting that $e_i = e_{i-1}-s_i$, it follows easily that
\[
a(e_i,e_i) = a(e_{i-1},e_{i-1})-a(s_i,s_i),
\]
and since the bilinear form $a$ is continuous and coercive
\[
\|e_i\|_{H^1(\Omega)}^2 + \|s_i\|_{H^1(\Omega)}^2 \leq C \|e_{i-1}\|_{H^1(\Omega)}^2,
\]
from which the result follows. \hspace*{\fill} \qed
\end{proof}

The following theorem and corollaries are our main results for the convergence of the multiscale Galerkin approximation. For the error analysis, we will need the norm
\begin{equation*}
\|u\|_{\Psi_j}^2 := \int\limits_{\mathbb{R}^d} |\widehat{u}(\boldsymbol{\omega})|^2 \left(1 + \delta_j^2 \|\boldsymbol{\omega}\|_2^2 \right) \mathrm{d}\boldsymbol{\omega}.
\end{equation*}
As in Lemma \ref{lemPhiDeltaHtauNormEquiv}, this norm satisfies
\begin{equation}
\label{eqnPsiEquivalance}
c_7 \|u\|_{\Psi_j} \leq \|u\|_{H^{1}(\mathbb{R}^d)} \leq c_8 \delta_j^{-1} \|u\|_{\Psi_{j}}.
\end{equation}

\begin{thm}
\label{thmMLGalerkinConvergence}
Let $\Omega \subseteq \mathbb{R}^d$ be a bounded domain with a Lipschitz boundary. Then for Algorithm \ref{AlgGalerkin} there exists a constant $\alpha_1>0$ such that
$$ \|Ee_i\|_{\Psi_{i+1}} \leq \alpha_1 \|Ee_{i-1}\|_{\Psi_i} \quad \mbox{for} \hspace{0.1in} i = 1,2,\ldots $$
where $Ee_i$ is the extension operator defined in Lemma \ref{lemExtensionOperator} applied to $e_i$. The constant $\alpha_1$ satisfies $\alpha_1<1$ if in Algorithm \ref{AlgGalerkin} $\nu$ is sufficiently small and $\mu$ is sufficiently large.
\end{thm}
\begin{proof}
Using \eqref{eqnPhiNorm} and \eqref{eqnScaledWendlandFourierTequiv}, we can write
$$\|Ee_i\|^2_{\Psi_{j+1}} \leq \frac{1}{c_1} \int\limits_{\mathbb{R}^d} \vert \widehat{Ee_i}(\boldsymbol{\omega}) \vert^2 \left(1 + \delta^2_{i+1} \| \boldsymbol{\omega} \|_2^2 \right) \ud \boldsymbol{\omega} =: \frac{1}{c_1}\left( I_1 + I_2 \right)$$
with
$$I_1 := \int\limits_{\|\boldsymbol{\omega}\|_2 \leq \frac{1}{\delta_{i+1}}} \vert \widehat{Ee_i}(\boldsymbol{\omega}) \vert^2 \left(1 + \delta^2_{i+1} \| \boldsymbol{\omega} \|_2^2 \right) \ud \boldsymbol{\omega},$$
$$ I_2 := \int\limits_{\|\boldsymbol{\omega}\|_2 \geq \frac{1}{\delta_{i+1}}} \vert \widehat{Ee_i}(\boldsymbol{\omega}) \vert^2 \left(1 + \delta^2_{i+1} \| \boldsymbol{\omega} \|_2^2 \right) \ud \boldsymbol{\omega}.$$
Now we consider the first integral where we can use that $\delta_{i+1} \|\boldsymbol{\omega}\|_2 \leq 1$ and then Lemmas \ref{lemGalerkinL2errorBound} and \eqref{eqnPsiEquivalance}.
\begin{eqnarray*}
I_1 &\leq& 2 \, \int\limits_{\|\boldsymbol{\omega}\|_2 \leq \frac{1}{\delta_{i+1}}} \vert \widehat{Ee_i}(\boldsymbol{\omega}) \vert^2 \ud \boldsymbol{\omega} \\
&\leq& 2 \, \|Ee_i\|^2_{L_2(\mathbb{R}^d)} \leq C \,\|e_i\|^2_{L_2(\Omega)} \leq C h_i^2 \,\|e_i\|^2_{H^1(\Omega)} \\
&\leq& C \, h_i^{2}\,\|e_{i-1}\|^2_{H^{1}(\Omega)} \leq C \left(\frac{h_i}{\delta_i} \right)^{2} \|Ee_{i-1}\|^2_{\Psi_i} \\
&=& C \nu^{-2} \, \|Ee_{i-1}\|^2_{\Psi_i},
\end{eqnarray*}
where we have also used Lemma \ref{lemGalerkinInterStageErrorBound}.
For $I_2$, since $\delta_{i+1} \|\boldsymbol{\omega}\|_2 \geq 1$, we have that
$$\left(1 + \delta^2_{i+1}\|\boldsymbol{\omega}\|_2^2 \right) \leq 2 \delta_{i+1}^{2} \|\boldsymbol{\omega}\|_2^{2}  \leq 2\delta_{i+1}^{2}\left(1 + \|\boldsymbol{\omega}_2^2 \right).$$
Then again using Lemma \ref{lemGalerkinInterStageErrorBound} shows that
\begin{eqnarray}
I_2 &\leq& 2\delta_{i+1}^{2} \|Ee_i\|^2_{H^{1}(\mathbb{R}^d)} \leq C \delta_{i+1}^{2} \|e_i\|^2_{H^1(\Omega)} \nonumber \\
&\leq& C \delta_{i+1}^{2} \|e_{i-1}\|^2_{H^1(\Omega)} \leq C \left(\frac{\delta_{i+1}}{\delta_i} \right)^{2} \|Ee_{i-1}\|^2_{\Psi_{i}} \label{eqnI2_2}\\
&\leq& C \mu^{2} \|Ee_{i-1}\|^2_{\Psi_{i}}. \nonumber
\end{eqnarray}
Combining our results for $I_1$ and $I_2$ and now writing $C_1$ and $C_2$ for the two constants appearing in the bounds of the expressions for $I_1$ and $I_2$ respectively, we have that
$$\|Ee_i\|^2_{\Psi_{i+1}} \leq \left( \nu^{-2} C_1/c_1  + \mu^2 C_2/c_1 \right) \, \|Ee_{i-1}\|^2_{\Psi_i},$$
and the result follows with
\begin{equation}
\alpha_1 := \left( \nu^{-2}C_1/c_1  + \mu^2 C_2/c_1 \right)^{1/2}. \label{eqnAlpha1}
\end{equation}
\hspace*{\fill} \qed
\end{proof}

\begin{cor}
\label{corGalerkinL2}
There exists a constant $C_3>0$ such that
\begin{equation}
\|u-\widetilde{u}_n\|_{L_2(\Omega)} \leq C_3 \alpha_1^n \|u\|_{H^{1}(\Omega)} \quad \mbox{for} \hspace{0.1in} n=1,2,\ldots
 \end{equation}
 Thus $\widetilde{u}_n$ resulting from Algorithm \ref{AlgGalerkin} converges linearly to $u$ in the $L_2$-norm if $\alpha_1 < 1$.
\end{cor}
\begin{proof}
Using Lemmas \ref{lemGalerkinL2errorBound} and \ref{lemPhiDeltaHtauNormEquiv} again, we can see that
\begin{eqnarray*}
\|u-\widetilde{u}_n\|_{L_2(\Omega)} &=& \|e_n\|_{L_2(\Omega)} \leq C h_n \|e_n\|_{H^1(\Omega)} \\
&\leq& Ch_n \|Ee_n\|_{H^{1}(\mathbb{R}^d)} \leq C \, h_n \, \delta_{n+1}^{-1} \|Ee_n\|_{\Psi_{n+1}} \\
&\leq& C\|Ee_n\|_{\Psi_{n+1}},
\end{eqnarray*}
since
\begin{equation*}
\frac{h_n}{\delta_{n+1}} = \frac{h_n}{\nu h_{n+1}} \leq \frac{1}{c \mu \nu}.
\end{equation*}
 Now we can apply Theorem \ref{thmMLGalerkinConvergence} $n$ times, and noting that $\widetilde{u}_0 = 0$, leads to
$$\|u-\widetilde{u}_n\|_{L_2(\Omega)} \leq C \alpha_1^n \|Eu\|_{\Psi_1} \leq C \alpha_1^n \|Eu\|_{H^{1}(\mathbb{R}^d)} \leq C \alpha_1^n \|u\|_{H^{1}(\Omega)}.$$ \hspace*{\fill} \qed
\end{proof}

\subsection{Condition number}

In this subsection, we present upper and lower bounds for the eigenvalues of the multiscale Galerkin algorithm. Since the Galerkin approximation matrix is symmetric and positive definite, we know that the condition number is given by
\begin{equation}
\kappa(\mathbf{A}) = \frac{\lambda_{\max}(\mathbf{A})}{\lambda_{\min}(\mathbf{A})} \label{eqnCondNumberMaxMinEig},
\end{equation}
where $\lambda_{\max}(\mathbf{A})$ and $\lambda_{\min}(\mathbf{A})$ denote the maximum and minimum eigenvalues of $\mathbf{A}$ as defined in \eqref{eqnAgalerkin}.

\begin{thm}
\label{thmGalCondNumber}
Let $\Phi$ be a positive definite kernel generating $H^{\tau}(\mathbb{R}^d)$ with $\tau > d/2$. Let $\Phi_i := \Phi(\cdot,\mathbf{x}_i)$ and assume that there exists a constant $c_6 > 0$ depending only on $\Phi$ and $\Omega$ such that
\begin{equation}
\alpha^T \left(\mathbf{F} - c_6 \mathbf{G}\right) \alpha \geq 0, \quad \forall \, \,\alpha \in \mathbb{R}^N, \label{eqnCondNumberAssumption}
\end{equation}
which means that $\mathbf{F} - c_6 \mathbf{G}$ is positive semi-definite and where
\begin{eqnarray*}
F_{i,j} &=& \left\langle \Phi_i,\Phi_j \right\rangle_{H^1(\Omega)}, \\
G_{i,j} &=& \left\langle \Phi_i,\Phi_j \right\rangle_{H^1(\mathbb{R}^d)}.
\end{eqnarray*}
Then the condition number of $\mathbf{A}$ can be bounded by
\begin{equation*}
\kappa(\mathbf{A}) \leq C \left(\frac{1}{q_X} \right)^{4\tau-2d},
\end{equation*}
where the constant $C$ is independent of the point set $X$.
\end{thm}
\begin{proof}
Since $v = \sum_{i=1}^N \alpha_i \Phi_i$ and with
with the coercivity of $a(\cdot,\cdot)$ and \eqref{eqnCondNumberAssumption}, we have that
\begin{eqnarray*}
a(v,v) &\geq& C \|v\|_{H^1(\Omega)}^2 = C \alpha^T \mathbf{F} \alpha^T \\
&\geq& C \alpha^T \mathbf{G} \alpha \geq C \lambda_{\min}(\mathbf{G}) \|\alpha\|_2^2,
\end{eqnarray*}
where $\lambda_{\min}(\mathbf{G})$ is the minimum eigenvalue of $\mathbf{G}$. From \cite{Wen98b}, we know that $\left\langle \Phi_i,\Phi_j \right\rangle_{H^1(\mathbb{R}^d)}$ is a radial function given by
\begin{equation*}
\Upsilon\left(\mathbf{x}_i,\mathbf{x}_j \right) := - \left( \Delta \Phi \right) \star \Phi(\mathbf{x}_i -\mathbf{x}_j) + \Phi \star \Phi(\mathbf{x}_i-\mathbf{x}_j),
\end{equation*}
where $\Delta$ again denotes the Laplace operator and $\star$ denotes convolution defined as $f \star g(x) := \int f(y) g(x-y) \, \mathrm{d}y.$ From \cite[Theorem 12.3]{Wen05} we know that we can use $\widehat{\Upsilon}$ to derive a lower bound on the minimum eigenvalue of $\mathbf{G}$. Then we have that
\begin{equation*}
\widehat{\Upsilon}(\mathbf{z}) = \left(1 + \vert \mathbf{z} \vert^2 \right) \widehat{\Phi}^2(\mathbf{z}).
\end{equation*}
From \cite[Theorem 10.35]{Wen05}, we know that $\widehat{\Phi}(\mathbf{z}) \geq C \|\mathbf{z}\|^{-2\tau}$ and hence $\widehat{\Upsilon}(\mathbf{z}) \geq C \|\mathbf{z}\|^{-4\tau}$. Then using \cite[Theorem 12.3]{Wen05} we reach
\[
\lambda_{\min}(\mathbf{G}) \geq Cq_X^{4\tau-d}.
\]
With the continuity of $a(\cdot,\cdot)$, \cite[Theorem 14.2]{Wen98} and the non-negativity of norms, we also  have the following bound on the maximum eigenvalue
\begin{eqnarray*}
a(v,v) &\leq& C \|v\|^2_{H^1(\Omega)} \leq C \|v\|^2_{H^1(\mathbb{R}^d)} = C \alpha^T \mathbf{G} \alpha \\
&\leq& CN q_X^{-d} \|\alpha\|_2^2.
\end{eqnarray*}
These two bounds, in conjunction with \eqref{eqnCondNumberMaxMinEig}, complete the proof.
\hspace*{\fill} \qed
\end{proof}
We will consider \eqref{eqnCondNumberAssumption} further with the scaled RBFs $\Phi_i := \Phi_{\delta}(\cdot,\mathbf{x}_i)$, where $\Phi$ is the $C^6$ Wendland function given by
\begin{equation}
\Phi(\mathbf{x}) = \left(1-\|\mathbf{x}\|\right)^8_+ \, \left(32\|\mathbf{x}\|^3 + 25\|\mathbf{x}\|^2 + 8\|\mathbf{x}\|+1 \right), \label{eqnWendC6}
\end{equation}
which is positive definite on $\mathbb{R}^2$ \citep{Wen05}. Now since the support of the radial basis functions is fixed, $\|\Phi_i\|_{H^1(\mathbb{R}^d)}$ is fixed and is independent of the point set $X$ and $\Omega$ and we can express this as
\begin{equation}
\|\Phi_i\|_{H^1(\mathbb{R}^d)}^2 = \int\limits_{0}^{2\pi} \int\limits_{0}^{\delta} r \left(\phi_{\delta}^2(r) + \delta^2\left(\frac{\mathrm{d}}{\mathrm{d}r} \phi_{\delta}(r) \right)^2 \right) \mathrm{d}r \, \mathrm{d}\theta.  \label{eqnNormPhiRdSquare}
\end{equation}

In the case of the unit square, which will be used for the numerical experiments in Section \ref{SectionNumericalExperiments}, note that $\|\Phi_i\|_{H^1(\Omega)}$ is minimised when the RBF centre is located on a corner. This can be easily verified since moving the centre in any direction (in $\Omega$) will keep the original area inside $\Omega$ and lead to additional area being inside $\Omega$ and the integrand is non-negative. Then we can then bound $\|\Phi_i\|^2_{H^1(\Omega)}$ by
\begin{eqnarray}
\|\Phi_i\|_{H^1(\Omega)}^2 &\geq& \int\limits_{0}^{\frac{\pi}{2}} \int\limits_{0}^{\delta} r \left(\phi_{\delta}^2(r) + \delta^2\left(\frac{\mathrm{d}}{\mathrm{d}r} \phi_{\delta}(r) \right)^2 \right) \mathrm{d}r \, \mathrm{d}\theta \nonumber \\
&=& \frac{1}{4}\|\Phi_i\|_{H^1(\mathbb{R}^d)}^2, \label{eqnNormPhiOmegaSquare}
\end{eqnarray}
where we have also used equations from \cite[Appendix D]{Fas07}.

This means that
\begin{equation*}
\rho_i := \frac{\|\Phi_i\|_{H^1(\mathbb{R}^d)}}{\|\Phi_i\|_{H^1(\Omega)}} \leq 2.
\end{equation*}
As a result, if $\delta \leq q_X$, which means that there is no overlap between the various RBFs and $\mathbf{F}$ and $\mathbf{G}$ are diagonal, or if $\mathbf{F}$ and $\mathbf{G}$ are diagonally dominant, then \eqref{eqnCondNumberAssumption} will hold since then we know that positive diagonal entries will ensure at least positive semi-definiteness. Whilst we have been unable to prove that \eqref{eqnCondNumberAssumption} holds in full generality for the unit square, it is supported by extensive numerical testing. The numerical experiments in Section \ref{SectionNumericalExperiments} also provide empirical evidence since Corollary \ref{corQuasiUniformCondNumber} holds, which depends on Theorem \ref{thmGalCondNumber}.

We have the following theorem on the condition number of the multiscale algorithm.

\begin{thm}
\label{thmGalerkinMLcondNumber}
Let $\Phi$ be a positive definite kernel generating $H^{\tau}(\mathbb{R}^d)$. Then the condition number of the Galerkin approximation matrices from Algorithm \ref{AlgGalerkin} can be bounded by
\begin{equation}
\kappa(\mathbf{A}) \leq C \left( \frac{\delta}{q_X} \right)^{4\tau-2d},
\end{equation}
with a constant $C>0$ independent of $X$ and of the scaling parameter $\delta$.
\end{thm}
\begin{proof}
At each level, we now introduce the dataset $X/\delta = \left\{\mathbf{x_1}/\delta,\ldots,\mathbf{x}_M/\delta \right\}$, which obviously has separation radius
\begin{equation*}
q_{X/\delta} = \frac{q_X}{\delta}
\end{equation*}
and since $a(\cdot,\cdot)$ is bilinear, the Galerkin approximation matrix at each level is
\begin{eqnarray*}
\mathbf{A}_{X,\delta} &=& \left( a(\Phi_{\delta}(\cdot,\mathbf{x_i}),\Phi_{\delta}(\cdot,\mathbf{x_j})) \right) \\
&=& \left( \delta^{-2d} a\left(\Phi \left(\frac{\cdot-\mathbf{x}_i}{\delta} \right),\Phi \left(\frac{\cdot-\mathbf{x}_j}{\delta} \right)\right) \right) = \delta^{-2d}\mathbf{A}_{X/\delta,1}.
\end{eqnarray*}
Then the result follows with Theorem \ref{thmGalCondNumber}. \hspace*{\fill} \qed
\end{proof}

\begin{cor}
\label{corQuasiUniformCondNumber}
If the datasets are quasi-uniform, which means that $h_j/q_j$ is bounded above by a constant, then the condition numbers of the Galerkin approximation matrices from Algorithm \ref{AlgGalerkin} are bounded above by a constant.
\end{cor}
\begin{proof}
Algorithm \ref{AlgGalerkin} takes $\delta_j = \nu h_j$ with a constant $\nu>1$. With the assumption of quasi-uniformity, $h_j \leq c q_j$ and the result follows with Theorem \ref{thmGalerkinMLcondNumber}. \hspace*{\fill} \qed
\end{proof}

We note that since we only require the bilinear form $a(\cdot,\cdot)$ to be continuous and coercive, these theorems on the condition number will apply for PDEs with Robin and/or Neumann boundary conditions as well as PDEs with Dirichlet boundary conditions.

\section{A nested multiscale algorithm} \label{SectionNestedML}

In this section we will consider another multiscale Galerkin algorithm that was proposed in \cite{Wen98b}. This essentially extends Algorithm \ref{AlgGalerkin} and hence we can consider a PDE either with or without Dirichlet boundary conditions. We refer to this as a \textit{nested} multiscale algorithm because it contains inner and outer iterations. We will also see that this has a connection to multigrid methods from the finite elements literature. The details are given in Algorithm \ref{AlgNestedGalerkin}.

\begin{algorithm}[!htbp]
\DontPrintSemicolon
{\normalsize
\KwData{\hspace{0.45in} $K$: number of outer levels \\ \hspace{0.9in} $n$: number of inner levels \\ \hspace{0.45in} $\{X_{\ell}\}_{\ell=1}^{n}$: the set of centres for each inner level $\ell$, with mesh \\ \hspace{0.45in} norms at each inner level given by $h_{\ell}$ satisfying \\ \hspace{0.45in} $c \mu h_{\ell} \leq h_{\ell+1} \leq \mu h_{\ell}$ with  fixed $\mu \in (0,1), c \in (0,1]$ and \\ \hspace{0.45in} $h_1$ sufficiently small \\ \hspace{0.45in} $\{\delta_{\ell}\}_{{\ell}=1}^n:$ the scale parameters to use at each inner level, \hspace{0.45in} \\ \hspace{0.45in} satisfying $\delta_{\ell} = \nu h_{\ell}$, $\nu$ a fixed constant.}

\Begin{
Set $\widetilde{u}_0=0$ \\
\For{$k =0,1,2,\ldots,K$}{
\For{$\ell=1,2,\ldots,n$}{
With the level-specific approximation subspace $V_{\ell} := \mbox{span}\left\{\Phi_{\ell}(\cdot,\mathbf{x}), \mathbf{x} \in X_{\ell}\right\}$ solve for the Galerkin approximation given by
\begin{equation*}
s_{kn + \ell} \in V_{\ell} \, : \, a(s_{kn + \ell},v) = \left\langle f,v \right\rangle_{L_2(\Omega)} - a(\widetilde{u}_{kn + \ell-1},v)
\end{equation*}
for all $v \in V_{\ell}$. Update the solution according to
\begin{eqnarray*}
\widetilde{u}_{kn + \ell} &=& \widetilde{u}_{kn + \ell-1} + s_{kn + \ell} \\
\end{eqnarray*}
}
}
}
\KwResult{Approximate solution at level $n$, $\widetilde{u}_{n(K+1)}$ \\ \hspace{0.45in} The error at level $n(K+1)$, $e_{n(K+1)} := u - \widetilde{u}_{n(K+1)}$.}
}
\caption{Nested multiscale Galerkin approximation \label{AlgNestedGalerkin}}
\end{algorithm}

From \cite{Wen98b} we have the following theorem regarding convergence.
\begin{thm}
\label{thmWenNestedConv}
Let $u^*$ denote the best approximation to $u$ from $V_1+\ldots+V_n$ with respect to the norm $\|\cdot\|_{H^1(\Omega)}$. Then there exists $\theta \in (0,1)$ such that
\begin{equation*}
\|u^* - \widetilde{u}_K\|_{H^1(\Omega)} \leq \theta^K \|u\|_{H^1(\Omega)},
\end{equation*}
where $\widetilde{u}_K$ is the approximation from Algorithm \ref{AlgNestedGalerkin}.
\end{thm}

Note however that this does not mean that we have linear convergence of the approximation from Algorithm \ref{AlgNestedGalerkin} to the true solution $u$. The convergence of the approximation from Algorithm \ref{AlgNestedGalerkin} to the true solution $u$ is given in the following theorem.

\begin{thm}
\label{thmNestGalConv}
Let $\Omega \subseteq \mathbb{R}^d$ be a bounded domain with Lipschitz boundary. Let $\Phi$ be a kernel generating $H^1(\mathbb{R}^d)$ and $\Phi_j$ be defined by \eqref{eqnDefnScaledKernel} with scale factor $\delta_j$. Then for Algorithm \ref{AlgNestedGalerkin} there exist constants $\alpha_2 >0, C_5 > 0, C_6 > 0$ such that
\begin{equation*}
\|u-\widetilde{u}_{n(K+1)}\|_{L_2(\Omega)}  \leq C_4 \alpha_1^{n+K(n-1)} \, \alpha_2^K \, \|u\|_{H^1(\Omega)}, \quad K=1,2,\ldots
\end{equation*}
with $\alpha_1$ given by \eqref{eqnAlpha1}.
\end{thm}
\begin{proof}
Since there are $K+1$ outer iterations (since the outer level index starts at 0) and $n$ inner iterations, we have $(K+1)n$ iterations in total, of which $K(n-1)+n$ iterations are with subsequently decreasing scale parameters for which we can use Theorem \ref{thmMLGalerkinConvergence}. The remaining $K$ iterations involve the subsequent error estimation for $K>0$ and $\ell=1$ since in this case, we have an increasing scale parameter for which Theorem \ref{thmMLGalerkinConvergence} does not apply. With the proof of Theorem \ref{thmMLGalerkinConvergence}, we can derive a similar result with increasing scale parameters. In this case, we need to change the right hand side of \eqref{eqnI2_2} and the following line to
\begin{equation*}
\left(\frac{\delta_{i+1}}{\delta_i} \right) \leq \mu^{-2},
\end{equation*}
and then we can define
\begin{equation}
\alpha_2 := \left( \nu^{-2}C_1/c_1  + \mu^{-2} C_2/c_1 \right)^{1/2}. \label{eqnAlpha2}
\end{equation}
Note that $\alpha_2 > \alpha_1$ since $\mu < 1$ by definition and the constants are all positive which will mean a lower rate of convergence as compared to Algorithm \ref{AlgGalerkin}. The remainder of the proof follows the same steps as the proofs of Corollary \ref{corGalerkinL2} and we leave the details to the reader. \hspace*{\fill} \qed
\end{proof}

Note that the original justification for proposing this nested multiscale algorithm was that the errors from Algorithm \ref{AlgGalerkin} appeared to be dominated by a global behaviour, suggesting the need to go back and fit on a coarse set of centres with a large support. This is a similar idea used in the multigrid method in the finite element literature \citep[Chapter 6.3]{BreS08}. As stated in \cite[Chapter 44.3]{Fas07}, this additional outer iteration is known from Kaczmarz iteration, which is frequently used in the multigrid literature as a smoother  \citep{McC92,KanK97}).

\section{Numerical experiments} \label{SectionNumericalExperiments}

In this section, we present the results from applying the multiscale and nested multiscale algorithms to various PDEs.

\subsection{Multiscale Algorithm} \label{subSectionMlGalExp}

In this subsection we consider two PDEs, the first without Dirichlet boundary conditions and the second with Dirichlet boundary conditions.

The first problem is the Helmholtz-like equation with natural boundary conditions:
\begin{eqnarray*}
-\Delta{u} + u &=& f \quad \mbox{in} \quad \Omega, \\
\frac{\partial}{\partial \mathbf{n}} u &=& 0 \quad \mbox{on} \quad \partial \Omega.
\end{eqnarray*}
We take $\Omega = [-1,1]^2$ and $f(x,y) = \cos(\pi x) \cos(\pi y)$. The outer unit normal vector is denoted by $\mathbf{n}$. The exact solution is given by $$u(x,y) = \frac{\cos(\pi x)\cos(\pi y)}{2\pi^2+1}.$$

We again use the $C^6$ Wendland radial basis function given by \eqref{eqnWendC6}. We used five levels for the approximation, with equally spaced point sets at each level. The number of points, $N$, and the fill distances, $h$, are given in Table \ref{tblMLGalerkinMeshNorms}. We note that the fill distances decrease by almost exactly one half at each level and hence we select $\mu = \frac{1}{2}$.
\begin{table}[!htbp]
\begin{centering}
\small
\begin{tabular}{|c|c|c|c|c|c|}
\hline
Level&1&2&3&4&5\\
\hline
$N$ & 25 & 81 & 289 & 1089 & 4225 \\
$h$ & 3.5e{-1} & 1.75e{-1} & 8.75e{-2} & 4.37e{-2} & 2.19e{-2} \\
\hline
\end{tabular} \caption{The number of uniform points used at each level and the associated fill distance for the multiscale Galerkin approximation example} \label{tblMLGalerkinMeshNorms}
\end{centering}
\end{table}
The $L_2$ and $L_{\infty}$ errors and condition numbers ($\kappa$) of the stiffness matrix are given in Table \ref{tblMLGalerkinExample1C6}. The $L_2$ error was estimated using Gaussian quadrature with a $300 \times 300$ tensor product grid of Gauss-Lobatto points and the $L_{\infty}$ error was estimated with the same tensor product grid.

\begin{table}[!htbp]
\begin{centering}
\small
\begin{tabular}{|c|c|c|c|c|c|}
\hline
Level&1&2&3&4&5\\
\hline
$\delta_j$ & 2 & 1 & 0.5 & 0.25 & 0.125 \\
$\|e_j\|_2$ & 8.000e{-4} & 2.145e{-4} & 1.059e{-4} & 7.009e{-5} & 5.178e{-5} \\
$\|e_j\|_{\infty}$ & 1.721e{-3} & 7.273e{-4} & 3.764e{-4} & 2.147e{-4} & 1.394e{-4}  \\
$\kappa_j$ & 1.608e{+3} & 3.125e{+3} & 4.159e{+3} &  4.576e{+3} & 4.710e{+3} \\
\hline
\end{tabular} \caption{The scaling factors, approximation errors and condition numbers  of the stiffness matrices for the multiscale Galerkin algorithm with a Neumann boundary condition.} \label{tblMLGalerkinExample1C6}
\end{centering}
\end{table}

The second example uses the Poisson problem
\begin{eqnarray*}
-\Delta{u} &=& f \quad \mbox{in} \quad \Omega, \\
 u &=& 0 \quad \mbox{on} \quad \partial \Omega.
\end{eqnarray*}
We take $\Omega = [-1,1]^2$ and $f(x,y) = \sin(\pi x) \cos\left(\frac{\pi}{2} y\right)$. The exact solution is given by $$u(x,y) = \frac{\sin(\pi x) \cos\left(\frac{\pi}{2} y\right)}{1.25\pi^2}.$$ We again use the $C^6$ Wendland function as the kernel, with the same 5 levels as for the first example. To verify that \eqref{eqnNitscheBound} holds, we first check for the basis functions. For the boundary norm, since $\Omega=[-1,1]^2$, without loss of generality, we consider the $x=-1$ boundary of the domain only. Then we have boundary integrals of the form \citep[Appendix D]{Fas07}
\begin{eqnarray*}
\| \nabla \Phi_{\delta} \cdot \mathbf{n} \|^2_{L_2(\partial \Omega)} &=& \int\limits_{-\delta}^{\delta} \left( \frac{\partial \phi_{\delta}}{\partial y} \right)^2  \mathrm{d}y \\
&=& \delta^2 \int\limits_{-\delta}^{\delta} \left( \frac{22y}{\delta^2}\left(\frac{16y^2}{\delta^2}+\frac{7y}{\delta}+1 \right)\left(1-\frac{y}{\delta} \right)^7 \right)^2  \mathrm{d}y \\
&=& \frac{603969552384\delta}{11305}.
\end{eqnarray*}

For the interior, we have
\begin{eqnarray*}
\| \nabla \Phi_{\delta} \|^2_{L_2(\Omega)} &=& \delta^2 \int\limits_{\Theta} \int\limits_{0}^{\delta} r \left(\frac{d}{\mathrm{d}r} \phi_{\delta}(r) \right)^2 \mathrm{d}r \, \mathrm{d}\theta \\
&=& \delta^2 \int\limits_{\Theta} \int\limits_{0}^{\delta} r \left( \frac{22r}{\delta^2} \left(\frac{16r^2}{\delta^2}+\frac{7r}{\delta}+1 \right)\left(1-\frac{r}{\delta} \right)^7 \right)^2 \mathrm{d}r \, \mathrm{d}\theta \\
&=& \frac{2453\delta^2}{4845} \int\limits_{\Theta} \mathrm{d}\theta,
\end{eqnarray*}
where $\Theta$ specifies the support of $\phi$ in $\Omega$ and hence this last expression is finite and does not depend on $\delta$.

In practical applications, we need to select a value of $\beta_N$ satisfying $\beta_N > 2C_N^2/\delta$. In \cite{GriS02}, it is proposed to estimate $C_N/\sqrt{\delta}$ as the maximum eigenvalue of the generalised eigenvalue problem,
\begin{equation}
\mathbf{B}\mathbf{v} = \lambda\mathbf{D}\mathbf{v},
\end{equation}
where
\begin{equation}
B_{ij} = \int\limits_{\partial \Omega} \left(\nabla \Phi_i \cdot \mathbf{n}\right)\left(\nabla \Phi_j \cdot \mathbf{n}\right),
\end{equation}
and
\begin{equation}
D_{ij} = \int\limits_{\Omega} \nabla \Phi_i \cdot \nabla \Phi_j,
\end{equation}
where $i$ and $j$ run over the indices of all the radial basis functions with support overlapping the boundary. The extra calculation involved in this step is not significant since the entries of $\mathbf{B}$ are required for the construction of the stiffness matrix and the set of centres overlapping the boundary will generally be small compared to the entire set of centres. The maximum eigenvalue can also be efficiently computed with a simultaneous Rayleigh-quotient minimisation method \citep{LonM80}.

\begin{table}[!htbp]
\begin{centering}
\small
\begin{tabular}{|c|c|c|c|c|c|}
\hline
Level&1&2&3&4&5\\
\hline
$\delta_j$ & 2 & 1 & 0.5 & 0.25 & 0.125\\
$\|e_j\|_2$ & 8.125e{-3} & 1.451e{-3} & 3.229e{-4} & 8.217e{-5} & 2.219e{-5}\\
$\|e_j\|_{\infty}$ & 1.057e{-2} & 2.350e{-3} & 6.504e{-4} & 1.963e{-4} & 5.962e{-5} \\
$\kappa_j$ & 5.633e{+5} & 1.001e{+6} & 8.056e{+5} & 4.572e{+5} & 2.374e{+5} \\
\hline
\end{tabular} \caption{The scaling factors, approximation errors and condition numbers of the stiffness matrices for the multiscale Galerkin algorithm with Dirichlet boundary conditions} \label{tblMLGalerkinExample2C6}
\end{centering}
\end{table}

The results are in Table \ref{tblMLGalerkinExample2C6} and support the theoretical findings above. We note that whilst \cite{Wen98b} did not find convergence after the third level with a similar algorithm, this may be due to the approximations used to calculate the integrals, rather than the algorithm itself. The potential for errors in integration to affect the performance of Galerkin techniques are well known \citep{Str73,BreS08}. To estimate the integrals, we used the MATLAB functions \texttt{quad2d} and \texttt{quad} with an absolute tolerance value of $1e^{-10}$. We also estimated the non-zero integration range both to speed up the calculations as well as to reduce numerical error which can result if for example, we integrate over the entire domain $[-1,1]^2$ whilst the function only has a very small support.

\subsection{Nested multiscale algorithm} \label{SectionNestedExample}

In this subsection, we consider the same example as in Section \ref{subSectionMlGalExp} however now with Algorithm \ref{AlgNestedGalerkin} with $K=2$ and $n=2$. We use the first two levels of the example described in Section \ref{subSectionMlGalExp} as the inner iteration. We also use the same kernel. Our choice of $K$ and $n$ leads to a 6 level algorithm. A similar example was considered in \cite[Section 5]{Wen98b}, however a lack of information regarding the exact approximation spaces used for the inner and outer level iterations means we have not been able to compare our results. The results from this 6 level nested algorithm are in Table \ref{tblnestedGal1}.

\begin{table}[!htbp]
\begin{centering}
\small
\begin{tabular}{|c|c|c|c|c|c|c|}
\hline
Level&1&2&3&4&5&6\\
\hline
$N$ & 25 & 81 & 25 & 81 & 25 & 81 \\
$\|e_j\|_2$ & 8.000e{-4} & 2.145e{-4} & 2.045e{-4} & 2.088e{-4} & 1.991e{-4} & 2.034e{-4} \\
$\|e_j\|_{\infty}$ & 1.721e{-3} & 7.274e{-4} & 4.325e{-4} & 7.002e{-4} & 4.165e{-4} & 6.742e{-4} \\
$\kappa_j$ & 1.608e{+3} & 3.125e{+3} & 1.608e{+3} &  3.125e{+3} & 1.608e{+3} &  3.125e{+3} \\
\hline
\end{tabular} \caption{The number of centres, scaling factors, approximation errors and condition numbers of the stiffness matrices for the nested multiscale Galerkin algorithm} \label{tblnestedGal1}
\end{centering}
\end{table}

The results indicate erratic convergence and approximation errors far inferior to those using Algorithm \ref{AlgGalerkin}. This is not surprising since Theorem \ref{thmWendNestConv} indicates convergence of our approximation to the best approximation to $u$ from $V_1+V_2$ whilst in Algorithm \ref{AlgGalerkin} our approximation is formed from $V_1+\ldots+V_5$.

\section{Analysis of convergence} \label{SectionAlphaAnalysis}

In this section we will focus on estimation of the convergence and verifying approximation orders. Similarly to \cite{Wen10}, we will also rewrite the convergence results in terms of fill distances, which is the usual form of convergence results for radial basis function approximations \citep{Wen05,Fas07}.

\subsection{Multiscale Galerkin algorithm}

We consider Algorithm \ref{AlgGalerkin} with $h_1 = \mu$ and $h_{j+1} = \mu h_j$. Since $\mu$ is a constant, we can rewrite \eqref{eqnAlpha1} as
\begin{equation*}
\alpha_1 = c_3 \mu.
\end{equation*}
Then with Corollary \ref{corGalerkinL2} we have that
\begin{equation}
\|e_n\|_{L_2(\Omega)} := \|u-\widetilde{u}_n\|_{L_2(\Omega)} \leq C h_n^{1-\sigma} \|u\|_{H^1(\Omega)},
\label{eqnEh}
\end{equation}
with
\begin{equation}
\sigma := -\log c_3 / \log \mu \label{eqnSigma}.
\end{equation}
Hence we can either express our convergence in terms of an exponent of $h_n$ or equivalently $\alpha_1^n$.

It is of interest that the error bounds do not depend on the kernel used for the approximation spaces. Typically with a kernel which generates $H^{\tau}(\mathbb{R}^d)$, we see error bounds proportional to $h^{\tau}$. Since our kernel for the error analysis generates $H^1(\mathbb{R}^d)$, we have $h_n^1$. Henceforth, we analyse the convergence in terms of $\alpha_1$. We can calculate estimates of $\alpha_1$, which we denote by $\widetilde{\alpha}_1$, as follows
\begin{equation*}
\widetilde{\alpha}_{1,n} := \frac{\|e_n\|_{L_2(\Omega)}}{\|e_{n-1}\|_{L_2(\Omega)}}.
\end{equation*}

\begin{table}[!htbp]
\begin{centering}
\small
\begin{tabular}{|c|c|c|}
\hline
 & $C^2$ WF & $C^6$ WF \\
\hline
$\widetilde{\alpha}_{1,2}$ & 0.128 & 0.268 \\
$\widetilde{\alpha}_{1,3}$ & 0.318 & 0.494 \\
$\widetilde{\alpha}_{1,4}$ & 0.507 & 0.662 \\
$\widetilde{\alpha}_{1,5}$ & 0.618 & 0.739 \\
\hline
\end{tabular} \caption{The estimated convergence rates $\widetilde{\alpha}_{1,n}$ using the results for the $L_2$ norm errors from the first example in Section \ref{subSectionMlGalExp} with the $C^2$ and $C^6$ Wendland functions.} \label{tblMLconv}
\end{centering}
\end{table}

We can see that the estimated values of $\alpha_1$ are higher with the $C^6$ Wendland function, which indicates that we should not necessarily expect faster convergence with a smoother Wendland function and consequently we should not expect an error bound proportional to $h^{\tau}$.

\subsection{Nested multiscale Galerkin algorithm}

In this subsection, we will focus on considering convergence of the nested multiscale Galerkin algorithm in terms of $\alpha_1$ and $\alpha_2$. Note that a bound for the error at level $n(K+1)$ in terms of the fill distance at level $n(K+1)$, such as in \eqref{eqnEh}, will not be possible here because the fill distance at level $n(K+1)$ only depends on $n$ and not on $K$. In other words, increasing $K$ has no effect on the final fill distance.

An additional benefit of considering the nested multiscale algorithm is more estimates of $\alpha_1$, particularly for repeated applications of the inner iterations (when $K>1$). Table \ref{tblMLconv2} gives the estimates of $\alpha_1$ and $\alpha_2$ from considering successive $L_2$ norm error estimates in Section \ref{SectionNestedExample}. Successive error estimates will be of the form
\begin{equation*}
\frac{\|e_{i}\|_{L_2(\Omega)}}{\|e_{i-1}\|_{L_2(\Omega)}}.
\end{equation*}
By definition of our nested multiscale algorithm, for $i=n+1,2n+1,\ldots$ we have an estimate for $\alpha_2$ and in all other cases, an estimate for $\alpha_1$.

Table \ref{SectionNestedExample} presents the estimated convergence rates $\widetilde{\alpha}_{1,n}$ and $\widetilde{\alpha}_{2,n}$ using the results for the $L_2$ norm errors from the example in Section \ref{SectionNestedExample}.

\begin{table}[!htbp]
\begin{centering}
\small
\begin{tabular}{|c|c|c|}
\hline
Level & $\widetilde{\alpha}_{1,n}$ & $\widetilde{\alpha}_{2,n}$  \\
\hline
2 & 0.268 & \\
3 &  & 0.953 \\
4 & 1.021 & \\
5 &  & 0.954 \\
6 & 1.022 &  \\
\hline
\end{tabular} \caption{The estimated convergence rates $\widetilde{\alpha}_{1,n}$ and $\widetilde{\alpha}_{2,n}$ using the results for the $L_2$ norm errors from the example in Section \ref{SectionNestedExample}.} \label{tblMLconv2}
\end{centering}
\end{table}

Interestingly, the difficulties with convergence are not due to $\alpha_2$ which is seen to be less than 1 in all cases. This is empirical evidence of the effectiveness of the smoothing nature of the inner iterations, in that after the inner iterations, the errors are again of a global nature and hence a return to a coarse grid is justified. Convergence is affected by the repeated application of the inner iterations for which $\widetilde{\alpha}_{1,n}$ is always greater than 1. Empirically, this appears to suggest that the more localised features have already been captured in the approximate solution.

This also implies that the angles between the approximation subspaces are close to zero since the linear convergence rate in Theorem \ref{thmWenNestedConv} can be bounded above by the angle between the subspaces. This is covered next.

\subsubsection{Angles between subspaces}

For our analysis of the convergence of the nested multiscale algorithm, we require the following definition of the angle between subspaces.
\begin{defn}
Let $\mathcal{H}_1$ and $\mathcal{H}_2$ be closed subspaces of a Hilbert space $\mathcal{H}$ with $U := \mathcal{H}_1 \cap \mathcal{H}_2$. Then the angle $\alpha$ between $\mathcal{H}_1$ and $\mathcal{H}_2$ is given by
\begin{equation*}
\cos \alpha = \sup \left\{\left\langle u,v \right\rangle: u \in \mathcal{H}_1 \cap U^{\bot}, v \in \mathcal{H}_2 \cap U^{\bot} \mbox{ and } \|u\|, \|v\| \leq 1 \right\}.
\end{equation*}
\end{defn}

It is well known \citep{SmiSW77,Deu90} that Algorithm \ref{AlgNestedGalerkin} converges linearly in the following sense.

\begin{thm}
\label{thmWendNestConv}
Let $u^{*}$ be the best approximation to $u$ from $V_1+\ldots+V_j$ with respect to $\|\cdot\|_{H^1(\Omega)}$. Let $\widetilde{u}$ be the approximation from Algorithm \ref{AlgNestedGalerkin}. Let $\alpha_j$ be the angle between $V_j$ and $\cap_{i=j+1}^k V_i$. Then
\begin{equation*}
\|u^{*} - \widetilde{u}\|_{H^1(\Omega)} \leq c^K \|u\|_{H^1(\Omega)},
\end{equation*}
where
\begin{equation*}
c^2 \leq 1 - \prod_{j=1}^{n-1} \sin^2 \alpha_j.
\end{equation*}
\end{thm}

This means that we need to estimate $\sin \alpha_j$ to obtain upper bounds for the convergence rate. We follow a similar approach to \cite{BeaLB00}. We firstly define modified sets of centres as $\widetilde{X}_1 = X_1$ and
\begin{equation*}
\widetilde{X}_i = X_i \backslash \bigcup_{j=1}^{i-1} X_j, \quad i\geq2,
\end{equation*}
and the corresponding approximation spaces as
\begin{equation*}
\widetilde{V}_i = \mbox{span}\left\{\Phi_{\delta_i}(\cdot,\mathbf{x}), \mathbf{x} \in \widetilde{X}_i \right\}.
\end{equation*}
Then we need to find the supremum of the inner product of $u \in \widetilde{V}_i$ and $v \in \mathcal{A}_{i+1}$ where
\begin{equation*}
\mathcal{A}_{i+1} =  \bigcup_{j=i+1}^K \widetilde{V}_j = \mbox{span}\left\{\Phi_{\delta_j}(\cdot,\mathbf{x}_j), \mathbf{x}_j \in \bigcup_{j=i+1}^K \widetilde{X}_j \right\},
\end{equation*}
 with $\|u\|=\|v\|=1$.  With the matrix $\mathbf{K}^{\{12\}}$ given by
\begin{equation*}
K^{\{12\}}_{i,j} = \left\langle \Phi_{\delta_i}(\cdot,\mathbf{x}_i), \Phi_{\delta_j}(\cdot,\mathbf{x}_j) \right\rangle, \mathbf{x}_i \in \widetilde{X}_i, \mathbf{x}_j \in \bigcup_{j=i+1}^K \widetilde{X}_j,
\end{equation*}
and with coefficient vectors $\mu$ and $\nu$ for $u$ and $v$ respectively, we seek the supremum of $\mu \mathbf{K}^{\{12\}} \nu$.
We also define matrices $\mathbf{K}^{\{1\}}$ and $\mathbf{K}^{\{2\}}$ as
\begin{equation*}
K^{\{1\}}_{ij} = \left\langle \Phi_{\delta_i}(\cdot,\mathbf{x}_i), \Phi_{\delta_j}(\cdot,\mathbf{x}_j)\right\rangle, \mathbf{x}_i \in \widetilde{X}_i, \mathbf{x}_j \in \widetilde{X}_i,
\end{equation*}
and
\begin{equation*}
K^{\{2\}}_{ij} = \left\langle \Phi_{\delta_i}(\cdot,\mathbf{x}_i), \Phi_{\delta_j}(\cdot,\mathbf{x}_j)\right\rangle, \mathbf{x}_i \in \bigcup_{j=i+1}^K \widetilde{X}_j, \mathbf{x}_j \in \bigcup_{j=i+1}^K \widetilde{X}_j.
\end{equation*}
Let the Cholesky decomposition of $\mathbf{K}^{\{1\}}$ be $\mathbf{L}_1^T \mathbf{L}_1$. This is well-defined since $\mathbf{K}^{\{1\}}$ is strictly positive definite and symmetric. Then $\|u\|^2 = \mu^T \mathbf{L}_1^T \mathbf{L}_1 \mu$ and letting $\boldsymbol{\gamma}_1 = \mathbf{L}_1 \mu$ gives $\|u\|^2 = \boldsymbol{\gamma}_1^T \boldsymbol{\gamma}_1$.
We can follow a similar approach with $\mathbf{K}^{\{2\}}$ which gives $\|v\|^2 = \boldsymbol{\gamma}_2^T \boldsymbol{\gamma}_2$ with $\mathbf{K}^{\{2\}} = \mathbf{L}_2^T \mathbf{L}_2$ and $\boldsymbol{\gamma}_2 = \mathbf{L}_2 \nu$. However in this case since $\mathbf{K}^{\{2\}}$ is the union of radial basis functions with (possibly) different scaling factors, we cannot be sure that $\mathbf{K}^{\{2\}}$ is positive definite. In our example, $\mathbf{K}^{\{2\}}$ was always positive definite and we do not dwell further on this. Sufficient conditions for an interpolation matrix constructed with several scaling factors to be positive definite can be found in \cite[(11)]{BozLRS04} which also requires the Fourier transform of a Wendland function from \cite{CheSW11} to compute a lower bound on the minimum eigenvalue as given in \cite[Theorem 12.3]{Wen05}.

Then we have
\begin{eqnarray*}
\left\langle u,v \right\rangle &=& \mu^T \mathbf{K}^{\{12\}} \nu \\
&=& \mu^T \mathbf{L}_1^T (\mathbf{L}_1^{-1})^T \mathbf{K}^{\{12\}} \mathbf{L}_2^{-1} \mathbf{L}_2 \nu \\
&=& \boldsymbol{\gamma}_1^T \mathbf{M} \boldsymbol{\gamma}_2,
\end{eqnarray*}
with $\mathbf{M} := (\mathbf{L}_1^{-1})^T \mathbf{K}^{\{12\}} \mathbf{L}_2^{-1}$. The supremum of the inner product is given by the largest singular value of $\mathbf{M}$. We denote this supremum by $\sin \widetilde{\alpha}_j$ and the results with $\{\widetilde{V}_i \}_{i=1}^5$ are in Table \ref{tblSinAlphaJ}. We note that since $\widetilde{X}_i \subseteq X_i$, $\sin \widetilde{\alpha}_i$ is a lower bound on $\sin \alpha_i$. We chose to estimate $\sin \widetilde{\alpha}_i$ since by removing nested centres from later levels, we had less difficulties with singular matrices.

\begin{table}[!htbp]
\begin{centering}
\small
\begin{tabular}{|c|c|c|c|c|}
\hline
i&1&2&3&4\\
\hline
$\sin \widetilde{\alpha}_i$ & 9.849e{-3} & 2.684e{-2}  & 4.153e{-2}  & 6.987e{-2} \\
\hline
\end{tabular} \caption{The estimates of $\sin \widetilde{\alpha}_i$ with the approximation spaces $\{\widetilde{V}_i \}_{i=1}^5$. } \label{tblSinAlphaJ}
\end{centering}
\end{table}

\section{Acknowledgments}

We used some datasets and code from \cite{Fas07} as well as Andrea Tagliasacchi's ``kd-tree for matlab" library, available at \url{www.mathworks.com/matlabcentral/fileexchange/21512}. This work was supported by the Australian Research Council. The authors thank Ian H. Sloan and Robert S. Womersley for helpful discussions.

\bibliographystyle{agsm}
\bibliography{AndrewChernihPhDthesis}{}

\end{document}